\documentclass[preprint,12pt]{elsarticle}
\usepackage{amsfonts}
\usepackage{amsmath}
\usepackage{amssymb}
\usepackage{graphicx}
\usepackage[tableposition=top]{caption}
\usepackage{subcaption}
\usepackage{float}
\usepackage{setspace}
\usepackage{placeins}

\graphicspath{{C:/Figures/}}
\newtheorem{theorem}{Theorem}
\newtheorem{proof}{Proof}

\newtheorem{proposition}{Proposition}
\newtheorem{remark}{Remark}

\numberwithin{equation}{section}
\journal{...................}

\input{tcilatex}

\begin{document}

\begin{frontmatter}
\title{Complete synchronization of the Newton--Leipnik reaction diffusion chaotic system}
\author[a]{Samir Bendoukha}
\author[b]{Salem Abdelmalek}
\address[a]{Electrical Engineering Department, College of Engineering at Yanbu, Taibah University, Saudi Arabia. E-mail address: sbendoukha@taibahu.edu.sa}
\address[b]{Department of Mathematics and Computer Science, Laboratory of Mathematics, Informatics and Systems (LAMIS), University of Larbi Tebessi, Tebessa, 12002 Algeria.E-mail address: salem.abdelmalek@univ-tebessa.dz}
\begin{abstract}
In this paper we investigate the reaction--diffusion system corresponding to
the Newton--Leipnik chaotic system originally developed to model the rigid
body motion through linear feedback (LFRBM). We develop a nonlinear
synchronization scheme for the proposed reaction--diffusion system and prove
its global stability in the local sense by means of the eigenvalues of the
Jacobian and global sense through an appropriate Lyapunov functional. A
numerical example is presented to illustrate the results of this study.
\end{abstract}
\end{frontmatter}

\section{Introduction\label{SecIntro}}

Chaotic dynamical systems have attracted a lot of attention in the last few
decades. The term chaos refers to the seemingly random behavior exhibited by
some deterministic dynamical systems. Although the rajectories of such a
system seem random, they are in fact completely deterministic and can be
exactly reproduced given the same initial conditions. The phase--space
trajectories of a chaotic system are extremely sensitive to variations in
the initial conditions. This interesting property has motivated the
application of this type of systems in a wide range of scientific and
engineering disciplines \cite%
{Kapitaniak2000,Banerjee2013,Curry2012,Aihara2012}. Interest in chaos became
more apparent after the revolutionary synchronization work \cite%
{Yamada1983,Yamada1984,Afraimovich1983,Pecora1990}, which demonstrated that
two chaotic systems with completely different initial conditions can be
forced to mimic one another and follow the same path.

Most of the studies dealing with chaos consider an ODE system where the
dependent variables represent the evolution of some scalar physical
quantities, such as the concentration of a substance or the motion of a
robot arm, over time. One of the first application of chaos was in modelling
and understanding the laminar and turbulent flow of fluids. In \cite%
{Cross1993}, the authors argued that since fluid, for instance, consists of
a continuum of hydrodynamic modes and thus should be described by a
spatially extended system of equations, or a reaction--diffusion system,
rather than and ODE\ system. In their work, they examined some of the
dynamics of such systems and considered as examples the complex
Ginzburg--Landau equation and the Kuramoto--Sivashinsky equation.

In \cite{Lai1994}, spatio--temporal chaotic systems were examined. The
authors showed that the general chaotic behavior is similar to the ODE case
in the sense that the system is extremely sensitive to changes in the intial
states as well as the system's parameters. Parekh et al. \cite{Parekh1997}
studied the control of autocatalytic reaction--diffusion system. They
devised a synchronization scheme and established the convergence of the
error by means of appropriate Lyapunov functionals. An interesting summary
of the chaos of reaction--diffusion systems is given in \cite{Zelik2007}.

In the last decade, numerous studies have been carried out on the
stabilization and synchronization of spatio--temporal chaotic systems. It
has been observed that Neural networks exhibit the dynamics of chaos. Neural
networks have many applications across a variety of fields including
communications and control. Studies considering the synchronization of
neural networks include \cite{Wang2007,Yu2011,Yang2013}. Recent studies have
also examined spatio--temporal chaos in other systems such as predator--prey
models \cite{Hu2015} and the FitzHugh--Nagumo model \cite%
{Zaitseva2016,Zaitseva2017}.

In this paper, we are interested in the synchronization of the
spatio--temporal Newton--Leipnik chaotic system \cite{Leipnik1981}, which
has a double strange attractor. The following section of this paper will
recall the conventional Newton--Leipnik and some of its main dynamics
including the equilibrium solutions and their asymptotic stability. Section %
\ref{SecSync} presents the proposed control laws for the complete
synchronization of a pair of Newton--Leipnik systems. The local and global
asymptotic stability of the resulting synchronization error system are
investigated through the eigenvalues of the Jacobian and the direct Lyapunov
method. Section \ref{SecApp} presents a numerical example whereby the
proposed control strategy is applied and the slave states are shown to
converge towards the master states uniformly in space. Finally, Section \ref%
{SecSum} summarizes the results of the study.

\section{The Newton--Leipnik ODE\ Model\label{SecModel}}

\subsection{System Model}

The use of differential equations to model the motion of a rigid object such
as a robot or an aircraft has been around for decades. In \cite{Leipnik1981}%
, the authors started with a simple model of a rigid body in $3$%
--dimensional space where the body's center of mass is taken as the origin.
They assumed that the density function of the body and its moments of
intertia are known. Assuming the rigid body is a jet, for instance, the
authors considered a stabilization scheme for the rigid body motion through
linear feedback (LFRBM) of the torques. Through simple manipulation and
normalization of the model, they were able to simplify it to the well known
Newton--Leipnik form given by%
\begin{equation}
\left \{
\begin{array}{l}
\frac{du_{1}}{dt}=-au_{1}+u_{2}+10u_{2}u_{3}, \\
\frac{du_{2}}{dt}=-u_{1}-0.4u_{2}+5u_{1}u_{3}, \\
\frac{du_{3}}{dt}=\alpha u_{3}-5u_{1}u_{2},%
\end{array}%
\right.  \label{2.1}
\end{equation}%
where $u_{i}$ are functions of time and $\alpha $ is a parameter. The
dynamical system (\ref{2.1}) is chaotic in the sense that it has a positive
Lyapunov exponent, meaning that for an extremely small perturbation in the
initial conditions, the system follows a new trajectory that diverges from
the previous one at an exponential rate. Many studies that examined the
Newton--Leipnik system have shown that subject to specific values of $a$ and
$\alpha $ such as $a=-0.4$ and $\alpha =0.175$, which were considered in
\cite{Leipnik1981}, the system has a strange attractor with two equilibria.
Figures \ref{Fig1} and \ref{Fig2} show the states and trajectories,
respectively, for the initial data%
\begin{equation}
\left( u_{1},u_{2},u_{3}\right) ^{T}=\left( 0.349,0,-0.3\right) ^{T}.
\label{2.2}
\end{equation}%
It is easy to see that the system has a double strange attractor, which is
an interesting property.

\begin{figure}[tbph]
\centering \includegraphics[width=3.5in]{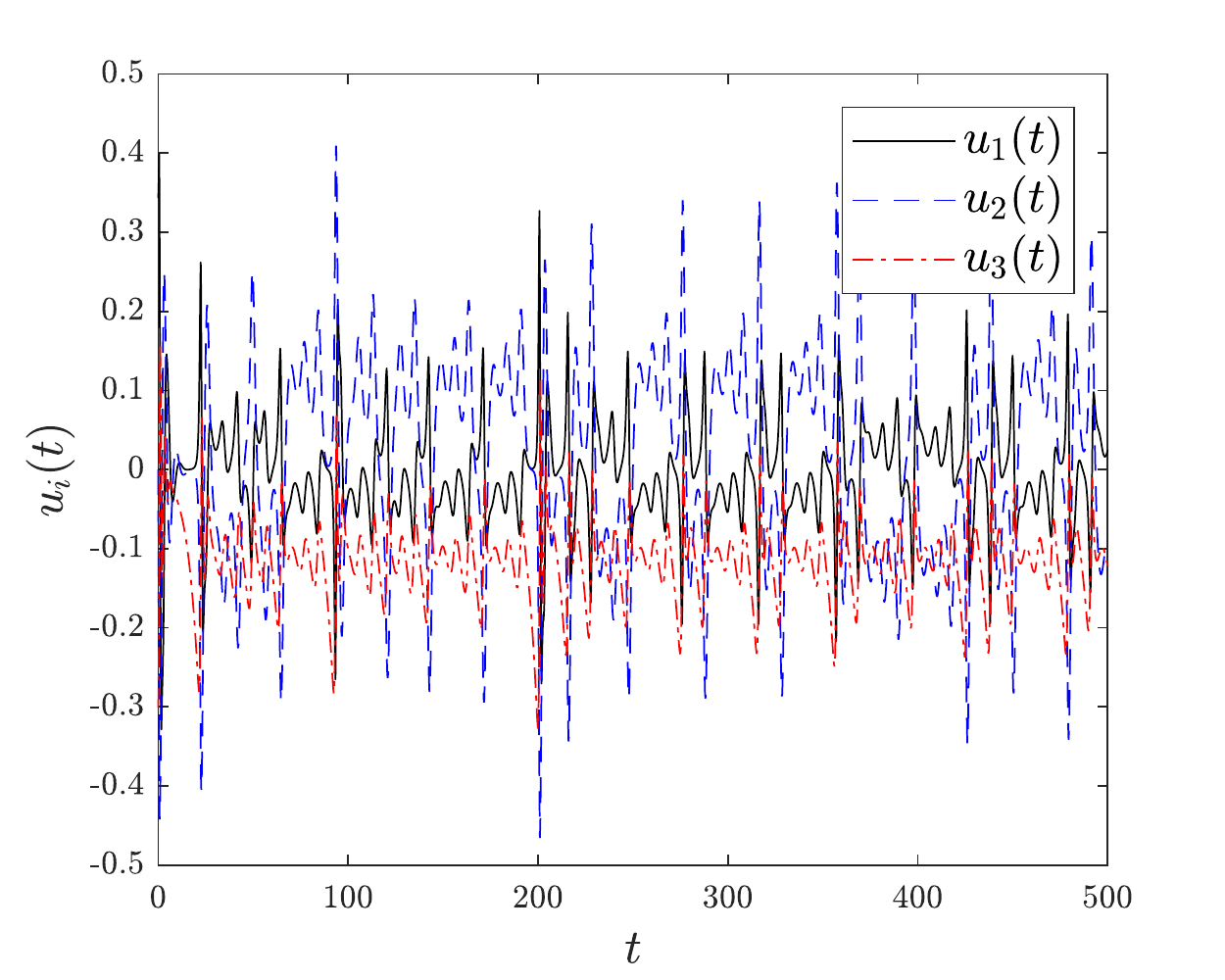}
\caption{States of the Newton--Leipnik system as functions of time for $%
\protect\alpha =0.175$ and initial conditions (\protect\ref{2.2}).}
\label{Fig1}
\end{figure}

\begin{figure}[tbph]
\centering \includegraphics[width=\textwidth]{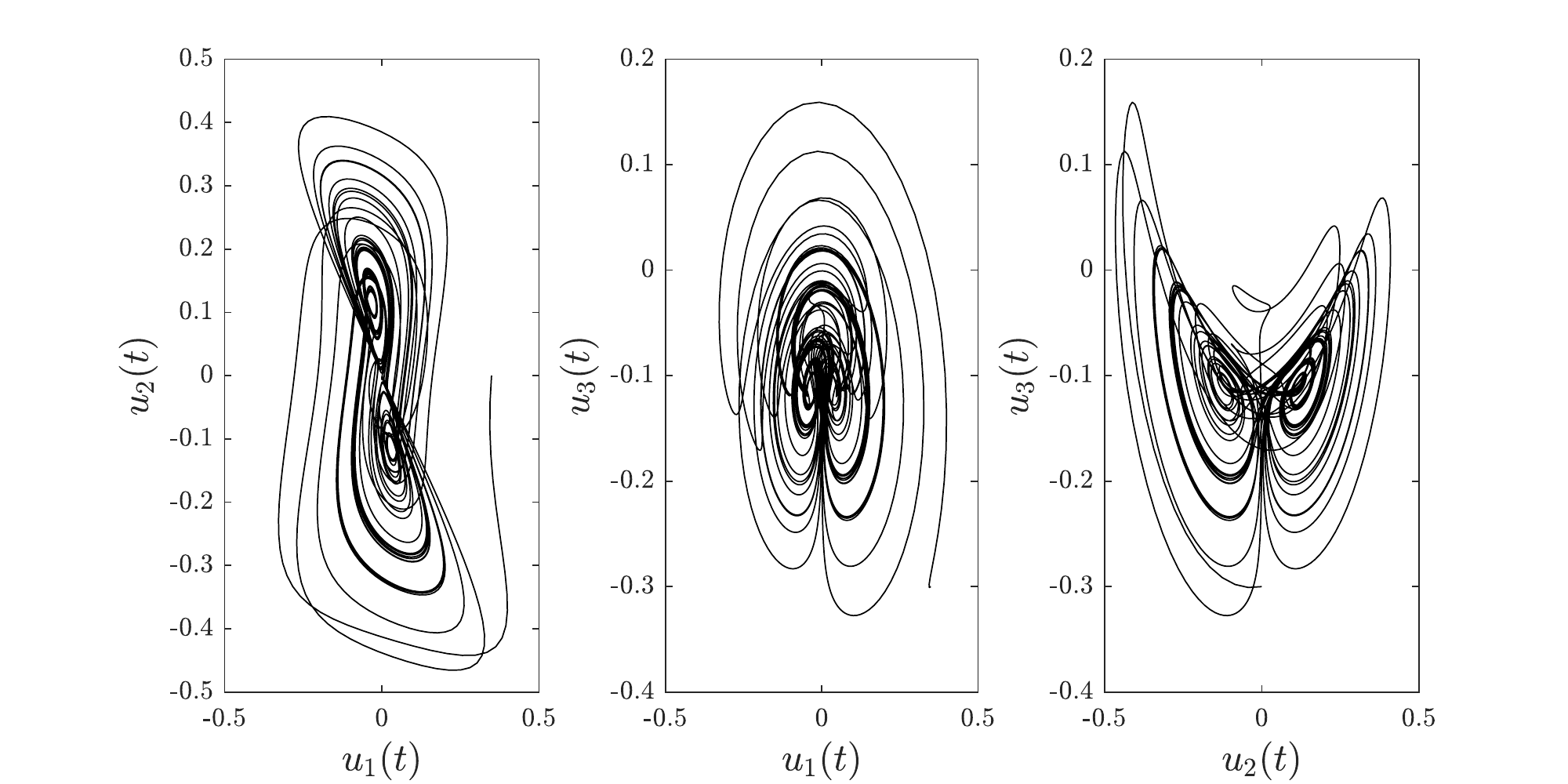}
\caption{Phase plots of the Newton--Leipnik system for $\protect\alpha %
=0.175 $ and initial conditions (\protect\ref{2.2}).}
\label{Fig2}
\end{figure}

\subsection{Dynamics of the ODE System}

In this section, we would like to study the main dynamics of the ODE chaotic
system (\ref{2.1}). We start by determining the equilibrium points and then
study the stability of the these points. Let us define the functions%
\begin{equation}
\left \{
\begin{array}{l}
f_{1}\left( u_{1},u_{2},u_{3}\right) =-au_{1}+u_{2}+10u_{2}u_{3} \\
f_{2}\left( u_{1},u_{2},u_{3}\right) =-u_{1}-0.4u_{2}+5u_{1}u_{3} \\
f_{3}\left( u_{1},u_{2},u_{3}\right) =\alpha u_{3}-5u_{1}u_{2}%
\end{array}%
\right.  \label{2.3}
\end{equation}%
The divergence of the vector field $f$ on $%
\mathbb{R}
^{3}$ is obtained as%
\begin{eqnarray*}
\text{div}f &=&\frac{\partial f_{1}}{\partial u_{1}}+\frac{\partial f_{2}}{%
\partial u_{2}}+\frac{\partial f_{3}}{\partial u_{3}} \\
&=&\alpha -a-0.4.
\end{eqnarray*}%
Let $\Omega $ be an arbitrary region in $%
\mathbb{R}
^{3}$ with a smooth boundary and let $\Omega \left( t\right) =\Phi
_{t}\left( \Omega \right) $, with $\Phi _{t}$ denoting the flow of field $f$%
. Also, let $V\left( t\right) $ be the volume of $\Omega \left( t\right) $.
It follows from Liouville's theorem that

\begin{eqnarray*}
\frac{dV\left( t\right) }{dt} &=&\int_{\Omega \left( t\right) }\left( \text{%
div}f\right) du_{1}du_{2}du_{3} \\
&=&\left( \alpha -a-0.4\right) V\left( t\right) .
\end{eqnarray*}%
The volume $V\left( t\right) $\ can, then, be obtained through simple
integration of the linear ODE yielding

\begin{equation}
V\left( t\right) =V\left( 0\right) e^{\left( \alpha -a-0.4\right) t}.
\label{2.4}
\end{equation}%
subject to
\begin{equation}
\alpha -a-0.4<0,  \label{2.5}
\end{equation}%
the volume in (\ref{2.4}) decays to zero as $t\rightarrow \infty $\ at an
exponential rate. It follows by definition that the system has a dissipative
nature. This means that the asymptotic motion of the system settles in all
cases onto a set that has a measure of zero. In other words, the system has
a strange attractor.

Let us, now, find the equilibrium points, which are the solutions of%
\begin{equation}
\left \{
\begin{array}{l}
-au_{1}+u_{2}+10u_{2}u_{3}=0, \\
-u_{1}-0.4u_{2}+5u_{1}u_{3}=0, \\
\alpha u_{3}-5u_{1}u_{2}=0.%
\end{array}%
\right.  \label{2.6}
\end{equation}%
We let $\alpha =0.175$ and $a=0.4$. We find that system (\ref{2.1}) has five
singular points%
\begin{eqnarray*}
O_{1} &=&\left(
\begin{array}{c}
0 \\
0 \\
0%
\end{array}%
\right) ,\ O_{2}=\left(
\begin{array}{c}
-0.03.154\,9 \\
0.122\,38 \\
-0.110\,31%
\end{array}%
\right) ,\ O_{3}=\left(
\begin{array}{c}
0.03154\,9 \\
-0.122\,38 \\
-0.110\,31%
\end{array}%
\right) , \\
O_{4} &=&\left(
\begin{array}{c}
0.238\,97 \\
0.030803 \\
0.210\,31%
\end{array}%
\right) ,\ O_{5}=\left(
\begin{array}{c}
-0.238\,97 \\
-0.030803 \\
0.210\,31%
\end{array}%
\right) .
\end{eqnarray*}

The matrix Jacobian matrix for the right--hand side of (\ref{2.1}) is easily
given by%
\begin{equation}
J=\left(
\begin{array}{ccc}
-0.4 & 1+10u_{3} & 10u_{2} \\
-1+5u_{3} & -0.4 & 5u_{1} \\
-5u_{2} & -5u_{1} & \alpha =0.175%
\end{array}%
\right) .  \label{2.7}
\end{equation}%
Substituting each of the five points in the Jacobian matrix and calculating
the corresponding eigenvalues leads to%
\begin{equation*}
\begin{array}{llll}
O_{1} & \rightarrow & \lambda _{1}=0.175 & \lambda _{2,3}=-0.4\pm i, \\
O_{2},O_{3} & \rightarrow & \lambda _{1}=-0.79997 & \lambda
_{2,3}=0.087484\pm 0.87526i, \\
O_{4},O_{5} & \rightarrow & \lambda _{1}=-0.79997 & \lambda
_{2,3}=0.087487\pm 1.2114i,%
\end{array}%
\end{equation*}%
Obviously, none of the Jacobians have eigenvalues with all negative real
parts. Hence, non of these points are asymptotically stable. This agrees
with the fact that the system has the Lyapunov exponents $k_{1}=0.1302$, $%
k_{2}=0$, and $k_{3}=-0.7537$ as reported in \cite{Wolf1985} and other
studies. Also, it is apparent from the phase plots in Figure \ref{Fig2} that
the two attractors are in fact points $O_{2}$ and $O_{3}$.

One of the major concerns when dealing with chaotic systems is their
control. The vast majority of their application relies on the concept of
synchronization, which aims to force the states $v_{i}$\ of a slave system
to follow the trajectories set out by the states $u_{i}$ of a master. In its
simplest form refered to as complete synchronization, the aim is to
introduce control paramaters into $v_{i}$\ to ensure $\lim_{t\rightarrow
\infty }\left \Vert u_{i}-v_{i}\right \Vert =0$ for all $i$. Some studies
have determined linear and nonlinear control laws to synchronize a pair of
Newton--Leipnik systems with different initial conditions including \cite%
{Qiang2008,Jovic2011}. In this paper, we aim to develop a control strategy
for the complete synchronization of the reaction--diffusion system
corresponding to the original Newton--Leipnik model as will be discussed in
the next section.

\section{Complete Synchronization Under Diffusion\label{SecSync}}

Let us, now, consider as master the reaction--diffusion system%
\begin{equation}
\left\{
\begin{array}{l}
\frac{du_{1}}{dt}-d_{1}\Delta u_{1}=-0.4u_{1}+u_{2}+10u_{2}u_{3},\text{ \ \
\ in }\mathbb{R}^{+}\times \Omega , \\
\frac{du_{2}}{dt}-d_{2}\Delta u_{2}=-u_{1}-0.4u_{2}+5u_{1}u_{3},\text{ \ \ \
\ \ in }\mathbb{R}^{+}\times \Omega , \\
\frac{du_{3}}{dt}-d_{3}\Delta u_{3}=\alpha u_{3}-5u_{1}u_{2},\text{ \ \ \ \
\ \ \ \ \ \ \ \ \ \ \ \ \ \ \ in }\mathbb{R}^{+}\times \Omega ,%
\end{array}%
\right.   \label{3.1}
\end{equation}%
where $\Omega $ is a bounded domain in $%
\mathbb{R}
^{n}$ with smooth boundary $\partial \Omega $ and $\Delta $ is the Laplacian
operator on $\Omega $. We assume non--negative continuous and bounded
initial data%
\begin{equation}
u_{i}\left( 0,x\right) =u_{i0}\left( x\right) ,\text{ \ }i=1,2,3\text{ \ \ \
\ \ \ \ \ \ \ in }\Omega ,  \label{3.1.1}
\end{equation}%
where $u_{i0}\left( x\right) \in C^{2}\left( \Omega \right) \cap C\left(
\overline{\Omega }\right) $, and homogoneous Neumann boundary conditions
\begin{equation}
\dfrac{\partial u_{i}}{\partial \nu }=0\ \ ,\text{ \ }i=1,2,3\text{\ \ \ \ \
on \ \ \ }\mathbb{R}^{+}\times \partial \Omega ,  \label{3.1.2}
\end{equation}%
with $\nu $ being the unit outer normal to $\partial \Omega $. The constants
$d_{1},d_{2},$ and $d_{3}$ are assumed to be strictly positive control
parameters. System (\ref{3.1}) is similar to (\ref{2.1}) but takes into
consideration the distribution of the functions $u_{i}\left( x,t\right) $ in
multi--dimensional space. The slave is defined in much the same way as%
\begin{equation}
\left\{
\begin{array}{l}
\frac{dv_{1}}{dt}-d_{1}\Delta v_{1}=-0.4v_{1}+v_{2}+10v_{2}v_{3}+\phi _{1},%
\text{\ in }\mathbb{R}^{+}\times \Omega , \\
\frac{dv_{2}}{dt}-d_{2}\Delta v_{2}=-v_{1}-0.4v_{2}+5v_{1}v_{3}+\phi _{2},%
\text{ \ \ in }\mathbb{R}^{+}\times \Omega , \\
\frac{dv_{3}}{dt}-d_{3}\Delta v_{3}=\alpha v_{3}-5v_{1}v_{2}+\phi _{3},\text{
\ \ \ \ \ \ \ \ \ \ \ \ \ \ \ \ in }\mathbb{R}^{+}\times \Omega ,%
\end{array}%
\right.   \label{3.2}
\end{equation}%
where $\phi _{i}$ are additive controllers to be defined later. The aim of
our control scheme is to find a closed form for $\phi _{i}$ as functions of $%
u=(u_{1},u_{2},u_{3})^{T},v=(v_{1},v_{2},v_{3})^{T}\in
\mathbb{R}
^{3}$ such that%
\begin{equation}
\lim_{t\rightarrow \infty }\left\Vert u-v\right\Vert _{\infty }=0
\label{3.3}
\end{equation}%
for any $t>0$. The synchronization error for the $i^{\text{th}}$ component
is defined as%
\begin{equation}
e_{i}\left( x,t\right) =v_{i}\left( x,t\right) -u_{i}\left( x,t\right) ,
\label{3.4}
\end{equation}%
leading to the following reaction--diffusion representation%
\begin{equation}
\left\{
\begin{array}{l}
\frac{de_{1}}{dt}-d_{1}\Delta
e_{1}=-0.4e_{1}+e_{2}+10v_{2}v_{3}-10u_{2}u_{3}+\phi _{1}, \\
\frac{de_{2}}{dt}-d_{2}\Delta
e_{2}=-e_{1}-0.4e_{2}+5v_{1}v_{3}-5u_{1}u_{3}+\phi _{2}, \\
\frac{de_{3}}{dt}-d_{3}\Delta e_{3}=\alpha
e_{3}-5v_{1}v_{2}+5u_{1}u_{2}+\phi _{3}.%
\end{array}%
\right.   \label{3.6}
\end{equation}%
We assume non-negative continuous and bounded initial data%
\begin{equation}
e_{i}\left( 0,x\right) =e_{i0}\left( x\right) ,\text{ \ }i=1,2,3\text{ \ \ \
\ \ \ \ \ \ \ in }\Omega ,
\end{equation}%
where $e_{i0}\left( x\right) \in C^{2}\left( \Omega \right) \cap C\left(
\overline{\Omega }\right) $, and homogoneous Neumann boundary conditions
\begin{equation}
\dfrac{\partial e_{i}}{\partial \nu }=0\ \ ,\text{ \ }i=1,2,3\text{\ \ \ \ \
on \ \ \ }\mathbb{R}^{+}\times \partial \Omega ,
\end{equation}%
The following theorem presents the main finding of this study.

\begin{theorem}
\label{Theo1}The master--slave pair (\ref{3.1})--(\ref{3.2}) is globally
synchronized subject to%
\begin{equation}
\left \{
\begin{array}{l}
\phi _{1}=-10v_{2}e_{3}+5u_{2}e_{3}, \\
\phi _{2}=-15u_{3}e_{1}, \\
\phi _{3}=-\left( \alpha +k\right) e_{3}.%
\end{array}%
\right.  \label{3.7}
\end{equation}
\end{theorem}

The proof of this theorem is extensive and involves establishing the local
and global asymptotic stability of the zero equilibrium of error system (\ref%
{3.6}). We will consider the two types of stability separately. Proposition %
\ref{Prop1} will show that the equilibrium is locally asymptotically stable
in the ODE sense. Proposition \ref{Prop2} will establish sufficient
conditions for the local asymptotic stability of the zero steady state.
Finally, Theorem \ref{Theo2} will show that subject to the same condition,
the zero steady state is globally asymptotically stable.

Before we can present these findings, let us substitute the control
parameters (\ref{3.7}) in (\ref{3.6}) and rewrite the resulting system in
matrix form yielding%
\begin{equation}
\frac{dE}{dt}-D\Delta E=AE,  \label{3.8}
\end{equation}%
where%
\begin{equation}
E=\left(
\begin{array}{c}
e_{1} \\
e_{2} \\
e_{3}%
\end{array}%
\right) ,\ D=\left(
\begin{array}{ccc}
d_{1} & 0 & 0 \\
0 & d_{2} & 0 \\
0 & 0 & d_{3}%
\end{array}%
\right) ,\ A=\left(
\begin{array}{ccc}
-0.4 & 1+10u_{3} & 5u_{2} \\
-1-10u_{3} & -0.4 & 5v_{1} \\
-5u_{2} & -5v_{1} & -k%
\end{array}%
\right) .  \label{3.9}
\end{equation}%
The following subsections will present the stability results.

\subsection{Local Stability}

It is well known from linear stability theory (see \cite{Casten1977}) that
an equilibrium point is locally asymptotically stable subject to all the
eigenvalues of the Jacobian matrix evaluated at that point having negative
real parts. First, in the absense of diffusion, (\ref{3.8}) becomes%
\begin{equation}
\frac{dE}{dt}=AE,  \label{3.10}
\end{equation}%
which has the point $\left( 0,0,0\right) ^{T}$ as its equilibrium. The
following proposition establishes sufficient conditions for the local
asymptotic stability of the zero equilibium of (\ref{3.10}).

\begin{proposition}
\label{Prop1}The solution $\left( 0,0,0\right) $ is a locally asymptotically
stable equilibrium for (\ref{3.10}) for some sufficiently large $k$.
\end{proposition}

\begin{proof}
Evaluating the Jacobian matrix (\ref{2.7}) at the zero solution yields
\begin{equation*}
J\left( 0,0,0\right) =\left(
\begin{array}{ccc}
-0.4 & 1+10u_{3} & 5u_{2} \\
-1-10u_{3} & -0.4 & 5u_{1} \\
-5u_{2} & -5u_{1} & -k%
\end{array}%
\right) .
\end{equation*}%
The system (\ref{3.10}) is locally asymptotically stable in the neighborhood
of equilibrium point $\left( 0,0,0\right) $ if the real parts of the
eigenvalues of $J\left( 0,0,0\right) $ are all negative. In order to ensure
that, we need to show that the determinant and trace of $J\left(
0,0,0\right) $ as well as the determinant of the second compound%
\begin{equation*}
J^{\left[ 2\right] }=\left(
\begin{array}{ccc}
-0.8 & 5u_{1} & 5u_{2} \\
-5u_{1} & -0.4-k & -1-10u_{3} \\
-5u_{2} & -1-10u_{3} & -0.4-k%
\end{array}%
\right) ,
\end{equation*}%
are all negative see \cite{Abdelmalek2016,Abdelmalek2018}. We start with
\begin{eqnarray*}
\det J\left( 0,0,0\right) &=&-10u_{1}^{2}-10u_{2}^{2}-20k\left(
5u_{3}^{2}+u_{3}\right) -1.16k, \\
&=&-10u_{1}^{2}-10u_{2}^{2}-100ku_{3}^{2}-20ku_{3}-1.16k, \\
&=&-10u_{1}^{2}-10u_{2}^{2}-\frac{1}{25}k\left(
2500u_{3}^{2}+500u_{3}+29\right) .
\end{eqnarray*}%
Since the discriminant of the polynomial $\left(
2500u_{3}^{2}+500u_{3}+29\right) $ is%
\begin{eqnarray*}
\Delta &=&\left( 500\right) ^{2}-4\left( 2500\right) \left( 29\right) \\
&=&-40\,000<0,
\end{eqnarray*}%
it is easy to see that $\det J\left( 0,0,0\right) $.\ Next, we look at the
trace, which is given by%
\begin{equation*}
\text{tr}J\left( 0,0,0\right) =-0.8-k,
\end{equation*}%
and is clearly negative. Lastly, and the determinant of $J^{\left[ 2\right]
} $ is given by%
\begin{eqnarray*}
\det J^{\left[ 2\right] }
&=&-0.8k^{2}-25ku_{1}^{2}-25ku_{2}^{2}-0.64k-10u_{1}^{2}-10u_{2}^{2} \\
&&\ \ \ \ -500u_{1}u_{2}u_{3}-50u_{1}u_{2}+80u_{3}^{2}+16u_{3}+0.672.
\end{eqnarray*}%
We observe that the coefficients of $k$ are all negative. Hence, there exist
sufficiently large values for $k$ such that $\det J^{\left[ 2\right] }<0$,
and thus the equilibrium $\left( 0,0,0\right) $\ becomes locally
asymptotically stable.
\end{proof}

Let us, now, include diffusion and assess the local stability of the zero
solution. In the presence of diffusion, the steady state solution satisfies
the following system%
\begin{equation}
\left \{
\begin{array}{l}
-d_{1}\Delta e_{1}=-0.4e_{1}+\left[ 1+10u_{3}\right] e_{2}+5u_{2}e_{3}, \\
-d_{2}\Delta e_{2}=-\left[ 1+10u_{3}\right] e_{1}-0.4e_{2}+5v_{1}e_{3}, \\
-d_{3}\Delta e_{2}=-5u_{2}e_{1}-5v_{1}e_{2}-ke_{3},%
\end{array}%
\right.  \label{3.11}
\end{equation}%
subject to the homogeneous Neumann boundary conditions
\begin{equation*}
\dfrac{\partial e_{1}}{\partial \nu }=\dfrac{\partial e_{2}}{\partial \nu }=%
\dfrac{\partial e_{3}}{\partial \nu }=0\text{ for all\ }x\in \partial \Omega
.
\end{equation*}

We denote the eigenvalues of the elliptic operator ($-\Delta $) subject to
the homogeneous Neumann boundary conditions on $\Omega $ by%
\begin{equation*}
0=\lambda _{0}<\lambda _{1}\leq \lambda _{2}\leq ...
\end{equation*}
We assume that each eigencalue $\lambda _{i}$ has multiplicity $m_{i}\geq 1$%
. We also denote the normalized eigenfunctions corresponding to $\lambda
_{i} $ by $\Phi _{ij},1\leq j\leq m_{i}$. It should be noted that $\Phi _{0}$
is a constant and $\lambda _{i}\rightarrow \infty $ as $i\rightarrow \infty $%
. The eigenfunctions and eigenvalues posess a number of interesting
properties including%
\begin{equation}
\begin{array}{lll}
-\Delta \Phi _{ij}=\lambda _{i}\Phi _{ij} & \text{in} & \Omega , \\
\frac{\partial \Phi _{ij}}{\partial \nu }=0 & \text{on} & \partial \Omega ,
\\
\int_{\Omega }\Phi _{ij}^{2}\left( x\right) dx=1. &  &
\end{array}
\label{3.12}
\end{equation}%
The following proposition establishes sufficient conditions for the local
asymptotic stability of the zero steady state soltuion.

\begin{proposition}
\label{Prop2}The constant steady state $\left( 0,0,0\right) $ is locally
asymptotically stable for (\ref{3.8}) if%
\begin{equation}
\frac{0.52+100u_{3}^{2}+20u_{3}}{1.6+2k}<d_{3}\lambda _{1}.  \label{3.13}
\end{equation}
\end{proposition}

\begin{proof}
Since (\ref{3.11}) has nonlinear reaction terms, we start by defining the
linearization operator%
\begin{equation}
L=\left(
\begin{array}{ccc}
-d_{1}\Delta -0.4 & 1+10u_{3} & 5u_{2} \\
-1-10u_{3} & -d_{2}\Delta -0.4 & 5u_{1} \\
-5u_{2} & -5u_{1} & -d_{3}\Delta -k%
\end{array}%
\right) .  \label{3.14}
\end{equation}%
Let $\left( \phi \left( x\right) ,\psi \left( x\right) ,\Upsilon \left(
x\right) \right) $ be an eigenfunction of $L$ corresponding to the
eigenvalue $\xi $, i.e. the pair satisfies%
\begin{equation*}
L\left( \phi \left( x\right) ,\psi \left( x\right) ,\Upsilon \left( x\right)
\right) ^{t}=\xi \left( \phi \left( x\right) ,\psi \left( x\right) ,\Upsilon
\left( x\right) \right) ^{t}.
\end{equation*}%
Alternatively, we can write%
\begin{equation*}
\left[ L-\xi I\right] \left( \phi \left( x\right) ,\psi \left( x\right)
\right) ^{t}=\left( 0,0,0\right) ^{t},
\end{equation*}%
leading to%
\begin{equation}
\left(
\begin{array}{ccc}
-d_{1}\Delta -0.4-\xi & 1+10u_{3} & 5u_{2} \\
-1-10u_{3} & -d_{2}\Delta -0.4-\xi & 5u_{1} \\
-5u_{2} & -5u_{1} & -d_{3}\Delta -k-\xi%
\end{array}%
\right) \left(
\begin{array}{c}
\phi \\
\psi \\
\Upsilon%
\end{array}%
\right) =\left(
\begin{array}{c}
0 \\
0 \\
0%
\end{array}%
\right) .  \label{3.15}
\end{equation}%
Using the factorizations%
\begin{equation*}
\phi =\sum_{0\leq i\leq \infty ,1\leq j\leq m_{i}}a_{ij}\Phi _{ij}\text{ },\
\psi =\sum_{0\leq i\leq \infty ,1\leq j\leq m_{i}}b_{ij}\Phi _{ij},\ \text{%
and }\Upsilon =\sum_{0\leq i\leq \infty ,1\leq j\leq m_{i}}c_{ij}\Phi _{ij},
\end{equation*}%
matrix equation (\ref{3.15}) can be formulated as%
\begin{equation*}
\sum_{0\leq i\leq \infty ,1\leq j\leq m_{i}}\left(
\begin{array}{ccc}
-d_{1}\lambda _{i}-0.4-\xi & 1+10u_{3} & 5u_{2} \\
-1-10u_{3} & -d_{2}\lambda _{i}-0.4-\xi & 5u_{1} \\
-5u_{2} & -5u_{1} & -d_{3}\lambda _{i}-k-\xi%
\end{array}%
\right) \left(
\begin{array}{c}
a_{ij} \\
b_{ij} \\
c_{ij}%
\end{array}%
\right) \Phi _{ij}=\left(
\begin{array}{c}
0 \\
0 \\
0%
\end{array}%
\right) .
\end{equation*}%
Disregarding the term $-\xi $, the stability of the steady state solution
relies on the eigenvalues of%
\begin{equation}
A_{i}=\left(
\begin{array}{ccc}
-d_{1}\lambda _{i}-0.4 & 1+10u_{3} & 5u_{2} \\
-1-10u_{3} & -d_{2}\lambda _{i}-0.4 & 5u_{1} \\
-5u_{2} & -5u_{1} & -d_{3}\lambda _{i}-k%
\end{array}%
\right) ,  \label{3.16}
\end{equation}%
having negative real parts. Deriving conditions for the negativity of the
eigenvalues is not easy for a $3\times 3$ matrix. Instead, we can examine
the trace and determinant of $A_{i}$ and the determinant of its second
additive compound $A_{i}^{\left[ 2\right] }$. We have%
\begin{equation*}
\text{tr}A_{i}=-\left( d_{1}+d_{2}+d_{3}\right) \lambda _{i}+\text{tr}J,
\end{equation*}%
which is clearly negativev given that $\text{tr}J$ is negative.\newline
The determinant of $A_{i}$\ is given by%
\begin{eqnarray}
\det A_{i} &=&-\left( d_{1}d_{2}d_{3}\right) \lambda _{i}^{3}  \notag \\
&&-\left( 0.4d_{1}d_{3}+0.4d_{2}d_{3}+kd_{1}d_{2}\right) \lambda _{i}^{2}
\notag \\
&&-\left(
25d_{1}u_{1}^{2}+25d_{2}u_{2}^{2}+100d_{3}u_{3}^{2}+20d_{3}u_{3}+1.16d_{3}+0.4kd_{1}+0.4kd_{2}\right) \lambda _{i}
\notag \\
&&+\det J.  \label{3.17}
\end{eqnarray}%
Obviously, it suffices for $\det J$ to be negative to achieve $\det A_{i}<0$%
. For $i=0$, we have%
\begin{equation*}
A_{0}=J=\left(
\begin{array}{ccc}
-0.4 & 1+10u_{3} & 5u_{2} \\
-1-10u_{3} & -0.4 & 5u_{1} \\
-5u_{2} & -5u_{1} & -k%
\end{array}%
\right) ,
\end{equation*}%
with determinant%
\begin{equation*}
\det J=-10u_{1}^{2}-10u_{2}^{2}-20k\left( 5u_{3}^{2}+u_{3}\right) -1.16k<0,
\end{equation*}%
leading to $\det A_{0}<0$ and consequently $\det A_{i}<0$. Now, let us look
at the determinant of $A_{i}^{\left[ 2\right] }$. We have%
\begin{equation*}
J^{\left[ 2\right] }=\left(
\begin{array}{ccc}
-0.8 & 5u_{1} & 5u_{2} \\
-5u_{1} & -0.4-k & -1-10u_{3} \\
-5u_{2} & -1-10u_{3} & -0.4-k%
\end{array}%
\right) ,
\end{equation*}%
with determinant
\begin{eqnarray*}
\det J^{\left[ 2\right] }
&=&-0.8k^{2}-25ku_{1}^{2}-25ku_{2}^{2}-0.64k-10u_{1}^{2} \\
&&-500u_{1}u_{2}u_{3}-50u_{1}u_{2}-10u_{2}^{2}+80u_{3}^{2}+16u_{3}+0.672.
\end{eqnarray*}%
Choosing k sufficiently large, we can guarantee that $\det J^{\left[ 2\right]
}$ in the same way is in the proof of Proposition \ref{Prop1}. The second
additive compound $A_{i}^{\left[ 2\right] }$ is of the form%
\begin{equation}
A_{i}^{\left[ 2\right] }=\left(
\begin{array}{ccc}
-\left( d_{1}+d_{2}\right) \lambda _{i}-0.8 & 5u_{1} & 5u_{2} \\
-5u_{1} & -\left( d_{1}+d_{3}\right) \lambda _{i}-0.4-k & -1-10u_{3} \\
-5u_{2} & -1-10u_{3} & -\left( d_{2}+d_{3}\right) \lambda _{i}-0.4-k%
\end{array}%
\right) ,  \label{3.20}
\end{equation}%
with the determinant%
\begin{eqnarray*}
\det A_{i}^{\left[ 2\right] } &=&-\left(
d_{1}^{2}d_{2}+d_{1}^{2}d_{3}+d_{1}d_{2}^{2}+2d_{1}d_{2}d_{3}+d_{1}d_{3}^{2}+d_{2}^{2}d_{3}+d_{2}d_{3}^{2}\right) \lambda _{i}^{3}
\\
&&-\left( \left( 0.4+k\right) \left( d_{1}^{2}+d_{2}^{2}\right)
+0.8d_{3}^{2}+\left( 1.6+2k\right) \left(
d_{1}d_{2}+d_{1}d_{3}+d_{2}d_{3}\right) \right) \lambda _{i}^{2} \\
&&+\left[ 0.52+100u_{3}^{2}+20u_{3}\right] \left( d_{1}+d_{2}\right) \lambda
_{i} \\
&&-\left( 1.6k\left( d_{1}+d_{2}+d_{3}\right) +k^{2}\left(
d_{1}+d_{2}\right) +0.64d_{3}+25u_{2}^{2}\left( d_{1}+d_{3}\right) \right. \\
&&\left. +25u_{1}^{2}\left( d_{2}+d_{3}\right) \lambda _{i}\right) +\det J^{
\left[ 2\right] }.
\end{eqnarray*}%
In order for $\det A_{i}^{\left[ 2\right] }$ to be negative, it suffices that%
\begin{equation*}
\left( 0.52+100u_{3}^{2}+20u_{3}\right) \left( d_{1}+d_{2}\right) \lambda
_{i}-\left( 1.6+2k\right) \left( d_{1}d_{3}+d_{2}d_{3}\right) \lambda
_{i}^{2}<0,
\end{equation*}%
which is guaranteed by (\ref{3.13}). This establishs the local asymptotic
stability of the zero steady state.
\end{proof}

\subsection{Global Aymptotic Stability}

Now that we have established the local asymptotic stability of the zero
solution, we can go ahead and apply the direct Lyapunov method to
investigate the global asymptotic stability. We propose the candidate
Lyapunov function%
\begin{equation}
V\left( t\right) =\frac{1}{2}\int_{\Omega }\left( e_{1}^{2}\left( x,t\right)
+e_{2}^{2}\left( x,t\right) +e_{3}^{2}\left( x,t\right) \right) dx.
\label{3.22}
\end{equation}%
The following theorem presents the main finding of the this study.

\begin{theorem}
\label{Theo2}Subject to (\ref{3.13}), the zero solution of (\ref{3.8}) is
globally asymptotically stable.
\end{theorem}

\begin{proof}
The derivative of $V\left( t\right) $ with respect to time is given by%
\begin{eqnarray*}
\frac{dV\left( t\right) }{dt} &=&\int_{\Omega }\left( \dot{e}_{1}e_{1}+\dot{e%
}_{2}e_{2}+\dot{e}_{3}e_{3}\right) dx. \\
&=&I+J,
\end{eqnarray*}%
where%
\begin{equation*}
I=d_{1}\int_{\Omega }e_{1}\Delta e_{1}dx+d_{2}\int_{\Omega }e_{2}\Delta
e_{2}dx+d_{3}\int_{\Omega }e_{3}\Delta e_{3}dx,
\end{equation*}%
and%
\begin{eqnarray*}
J &=&\int_{\Omega }\left[ -0.4e_{1}^{2}+\left[ 1+10u_{3}\right]
e_{1}e_{2}+5u_{2}e_{1}e_{3}\right] dx \\
&&+\int_{\Omega }\left[ -\left[ 1+10u_{3}\right]
e_{1}e_{2}-0.4e_{2}^{2}+5v_{1}e_{2}e_{3}\right] dx \\
&&+\int_{\Omega }\left[ -5u_{2}e_{1}e_{3}-5v_{1}e_{2}e_{3}-ke_{3}^{2}\right]
dx.
\end{eqnarray*}%
Simple manipulation of $I$ and $J$\ yields%
\begin{equation*}
I=-d_{1}\int_{\Omega }\left \vert \nabla e_{1}\right \vert
^{2}dx-d_{2}\int_{\Omega }\left \vert \nabla e_{2}\right \vert
^{2}dx-d_{3}\int_{\Omega }\left \vert \nabla e_{3}\right \vert ^{2}dx<0,
\end{equation*}%
and%
\begin{equation*}
J=-0.4\int_{\Omega }e_{1}^{2}dx-0.4\int_{\Omega }e_{2}^{2}dx-k\int_{\Omega
}e_{3}^{2}dx<0.
\end{equation*}%
Hence, the derivative $\frac{dV\left( t\right) }{dt}$ is negative
semi--definite on $%
\mathbb{R}
^{3}$. Consequently, we can say that the synchronization error vector $%
E\left( x,t\right) $ is globally bounded, i.e%
\begin{equation*}
E\left( x,t\right) =\left[ e_{1}\left( x,t\right) ,e_{2}\left( x.t\right)
,e_{3}\left( x,t\right) \right] ^{T}\in L_{\infty }.
\end{equation*}%
Using Barbalat's lemma, we can conclude that $E\left( x,t\right) \rightarrow
0$ exponentially as $t\rightarrow \infty $ for all initial conditions $%
E\left( x,0\right) \in
\mathbb{R}
^{3}$. This concludes the proof of Theorem \ref{Theo2}.
\end{proof}

\begin{remark}
In Proposition \ref{Prop2}, we showed that subject (\ref{3.13}) the
synchronization error (\ref{3.4}) converges towards zero for a sufficiently
large $k$ yielding local stability everywhere. Note that in Theorem \ref%
{Theo2} above,\ the global asymptotic stability is established for any $k>0$%
, which implies that there exists a point in time $t_{0}$\ such that for any
$t>t_{0}$, the zero solution is locally stable for all $k>0$.
\end{remark}

\section{Numerical Example\label{SecApp}}

Consider the same parameters used in Section \label{SecMod} for the ODE
Newton--Leipnik system, which were chosen as $a=-0.4$ and $\alpha =0.175$.
The initial conditions for the master and slave systems are given by%
\begin{equation}
\left \{
\begin{array}{l}
u_{1}\left( x,0\right) =0.349\times \left( 1+0.3\cos \left( \frac{\pi }{2}%
x\right) \right) , \\
u_{2}\left( x,0\right) =0, \\
u_{3}\left( x,0\right) =-0.3\times \left( 1+0.3\cos \left( \frac{\pi }{2}%
x\right) \right) ,%
\end{array}%
\right.  \label{4.1}
\end{equation}%
and%
\begin{equation}
\left \{
\begin{array}{l}
v_{1}\left( x,0\right) =0.7\times \left( 1+0.3\cos \left( \frac{3\pi }{5}%
x\right) \right) , \\
v_{2}\left( x,0\right) =0.15\times \left( 1+0.3\cos \left( \frac{2\pi }{5}%
x\right) \right) , \\
v_{2}\left( x,0\right) =0.7\times \left( 1+0.3\cos \left( \frac{7\pi }{10}%
x\right) \right) ,%
\end{array}%
\right.  \label{4.2}
\end{equation}%
respectively. Note that the cosine terms in (\ref{4.1}) and (\ref{4.2}) were
added with the aim of introducing spatial non--homogeneity.

A Matlab simulation was performed using the implicit finite difference
method with zero Neumann boundaries. The slave system was equipped with the
control laws specified in (\ref{3.7}). The resulting master and slave states
are depicted in Figure \ref{Fig3}. Figure \ref{Fig4} shows the error defined
in (\ref{3.4}). It is clear that the error decays to zero in sufficient
time, implying that the master--slave pair is globally synchronized. The
phase plots taken at the particular point $x=5$ in one--dimensional space
are shown in Figure \ref{Fig5}. The slave states converge towards the master
states.

\begin{figure}[tbph]
\centering \includegraphics[width=5.5in]{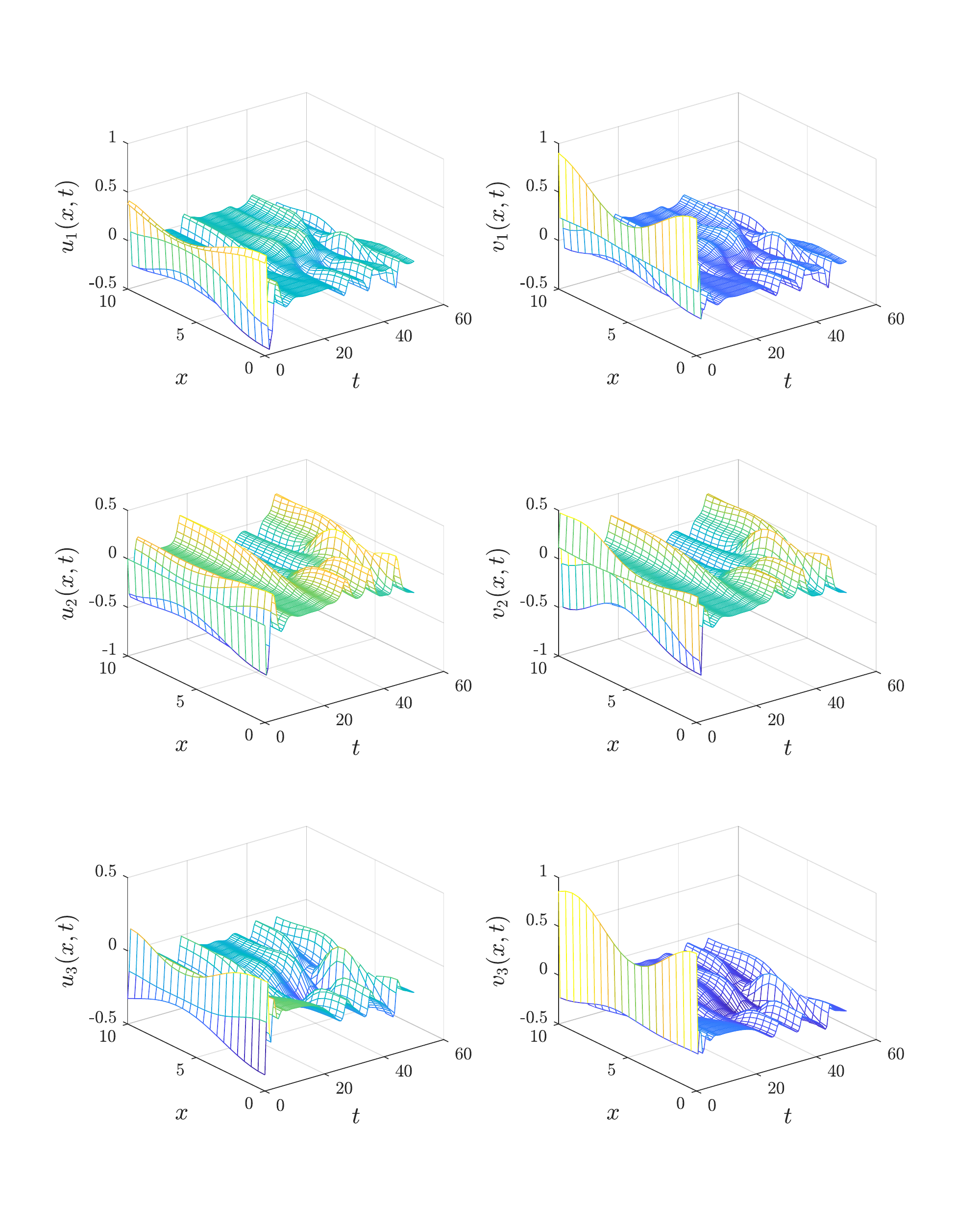}
\caption{The master states (right) and slave states (left) as a function of
time and space with the initial conditions in (\protect\ref{4.1}) and (%
\protect\ref{4.2}), respectively.}
\label{Fig3}
\end{figure}

\begin{figure}[tbph]
\centering \includegraphics[width=3.5in]{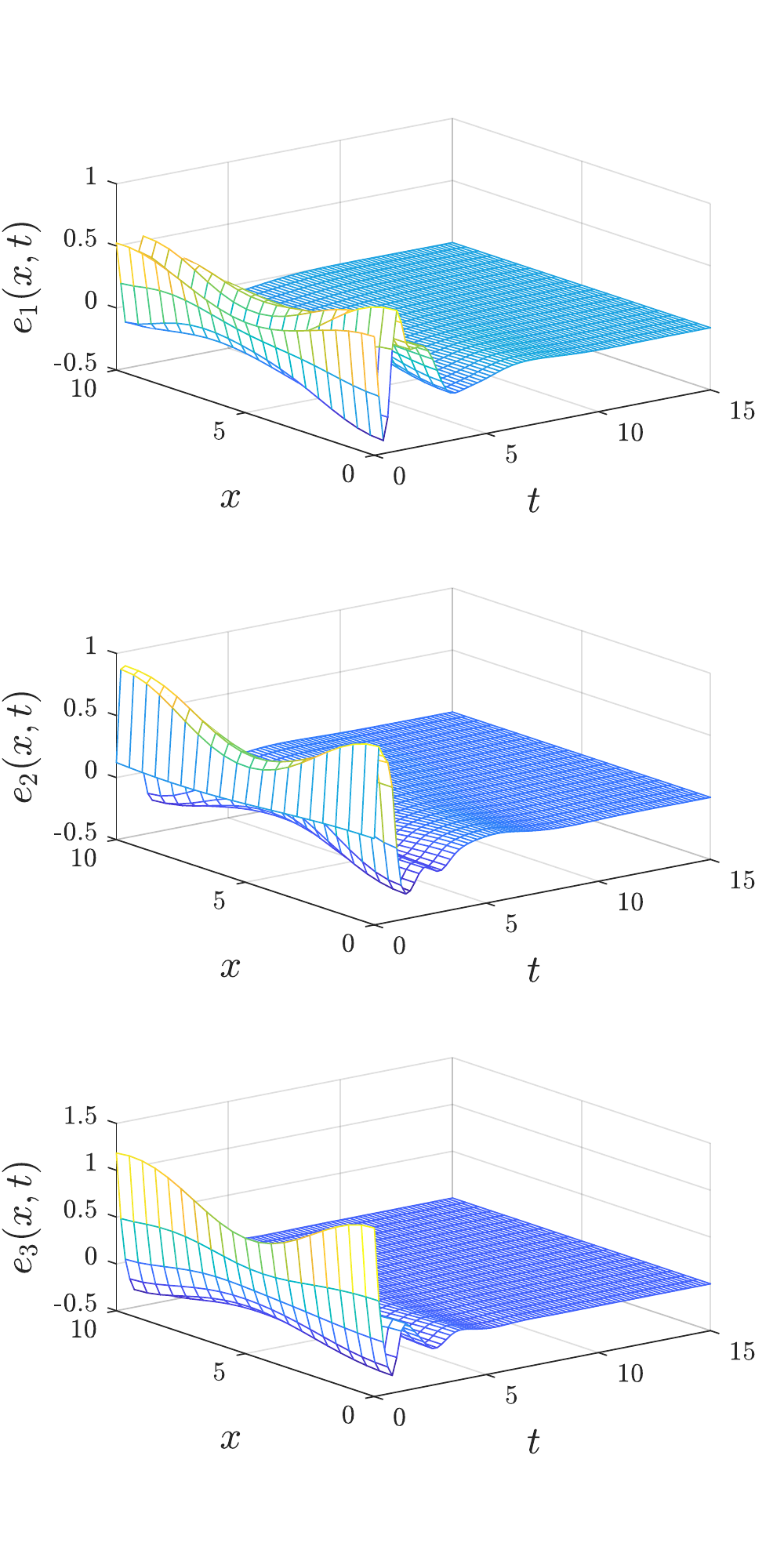}
\caption{The synchronization error as a function of time and space with the
initial conditions in (\protect\ref{4.1}) and (\protect\ref{4.2}),
respectively.}
\label{Fig4}
\end{figure}

\begin{figure}[tbph]
\centering \includegraphics[width=5.5in]{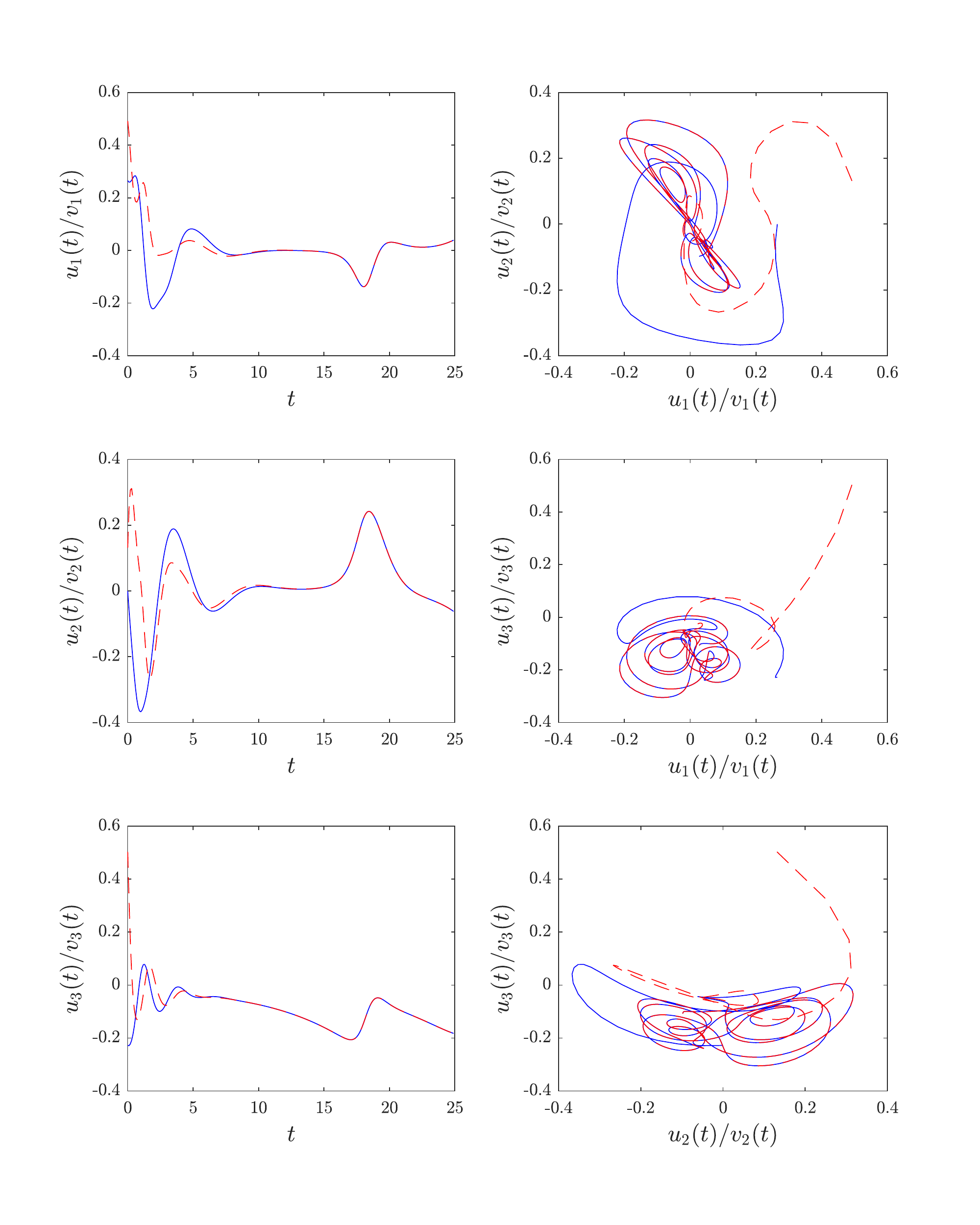}
\caption{Phase plots of the Newton--Leipnik reaction--diffusion system taken
at a specific point in one--dimensional space ($x=5$) for initial conditions
(\protect\ref{4.1}) and (\protect\ref{4.2}).}
\label{Fig5}
\end{figure}

\section{Concluding Remarks}

\label{SecSum}

In this paper, we studied the Newton--Leipnik chaotic system originally
developed to model the rigid body motion through linear feedback (LFRBM).
The Newton--Leipnik has one positive Lyapunov exponent yielding a chaotic
behavior in phase--space for certain values of the parameters. We have
recalled some of the dynamics of the ODE model as reported in the literature
including the equilibrium solutions and their stability. We then proposed a
master--slave configuration of reaction--diffusion Newton--Leipnik type
systems and proposed a control strategy guaranteeing complete
synchronization globally. In order to prove this result, we derived the
synchronization error reaction--diffusion system and studied the asymptotic
stability of its zero solutions. We showed that the zero steady state is
both locally and globally asymptotically stable by means of conventional
stability theory including the Lyapunov direct method. A numerical example
was considered to show the chaotic behavior of the system when diffusion is
considered and established the synchronization of the master and slave
systems using the proposed controls.

\end{document}